\documentclass[11pt]{article} 
\usepackage{amsmath,amsthm} 
\usepackage{amssymb,latexsym}
\usepackage[T1]{fontenc}
\usepackage{authblk}
\usepackage{lipsum}
\usepackage{hyperref}
\usepackage{lipsum} 
\usepackage{lineno}
\modulolinenumbers[1]
\newtheorem{theorem}{Theorem}
	\newtheorem{question}{Question}[section]
\newtheorem{thm}[theorem]{Theorem}
\newtheorem{cor}[theorem]{Corollary}
\newtheorem{lem}[theorem]{Lemma}
\theoremstyle{Proposition}
\newtheorem{prop}[theorem]{Proposition}
\theoremstyle{definition}
\newtheorem{deff}[theorem]{Definition}
\newtheorem{eg}[theorem]{Example}

\newtheorem{rem}[theorem]{Remark}

\begin{document}
	
	\title{A Study of $S$-Primary Decompositions}
	
	\author[1]{Tushar Singh}
	\author[2]{Ajim Uddin Ansari}
	\author[3]{ Shiv Datt Kumar }

	\affil[1, 3]{\small Department of Mathematics, Motilal Nehru National Institute of Technology Allahabad, Prayagraj 211004, India \vskip0.01in Emails: sjstusharsingh0019@gmail.com, tushar.2021rma11@mnnit.ac.in, sdt@mnnit.ac.in}
	\vskip0.05in
	\affil[2]{\small Department of Mathematics, CMP Degree College, University of Allahabad, Prayagraj-211002, India \vskip0.01in Email: ajimmatau@gmail.com}
	\date{}
	\maketitle
	\hrule
	
	 \begin{abstract}
	 	\noindent
	 	Let $R$ be a commutative ring with identity, and let $S \subseteq R$ be a multiplicative set. An ideal $Q$ of $R$ (disjoint from $S$) is said  to be $S$-primary if there exists an $s\in S$ such that for all $x,y\in R$ with $xy\in Q$, we have $sx\in Q$ or $sy\in rad(Q)$. Also, we say that an ideal  of $R$ is $S$-primary decomposable or has an $S$-primary decomposition if it can be written as a finite intersection of $S$-primary ideals. In this paper, first we provide an example of an $S$-Noetherian ring in which an ideal does not have a  primary decomposition. Then our  main aim of this paper is to  establish the existence and uniqueness of $S$-primary decomposition in $S$-Noetherian rings as an extension of a historical theorem of Lasker-Noether.
	\end{abstract}
	\hrule
	
	\smallskip

	\section{Introduction} \label{s1}
The theory of Noetherian rings has been playing an important role in the development of structure theory of commutative rings. One of the roots of this theory is the historical article \cite{ne21} by Noether in 1921. Recall that a ring is called  Noetherian  if it satisfies ascending chain condition on ideals. In the past few decades, several generalizations of Noetherian rings have been extensively studied by many authors because of its importance (see \cite{ah18},  \cite{sh16},  \cite{sh15}, \cite{ad02}, \cite{bj16}, and \cite{lw15}). 
As one of its crucial generalizations,  Anderson and Dumitrescu \cite{ad02} introduced $S$-Noetherian rings. 
A commutative ring $R$ with identity is called $S$-Noetherian,
where $S \subseteq R$ is a given multiplicative set, if for each ideal $I$ of
$R$, $sI \subseteq J \subseteq I$ for some $s \in  S$ and some finitely generated ideal $J$. 
Theory of primary decomposition, considered as a generalization of the factorization of an integer $n \in \mathbb{Z}$ into a product of prime powers, initiated by Lasker-Noether \cite{le05, ne21}  in their abstract treatment of commutative rings. Lasker-Noether proved that  in a commutative Noetherian ring, every ideal can be decomposed as an intersection, called primary decomposition, of finitely many primary ideals (popularly known as the Lasker-Noether decomposition
theorem). Due to its significance, this theory quickly grew as one of the basic tools of  commutative algebra and algebraic geometry. It gives an algebraic foundation for decomposing an algebraic variety into its irreducible components. Recently in \cite{me22}, Massaoud introduced the concept of $S$-primary ideals as a proper generalization of primary ideals. Let $R$ be a commutative ring with identity and $S \subseteq R$ a multiplicative set. A proper ideal $Q$ (disjoint from $S$) of $R$ is said to be $S$-primary if there exists an $s \in S$ such that for all $a,b \in R$ if $ab \in Q$, then $sa \in Q$ or $sb \in rad(Q)$. The author \cite{me22} investigated several properties of this class of ideals and showed that $S$-primary ideals enjoy analogue of many properties of primary ideals.
Given the significance of primary decomposition in Noetherian rings, a natural question arises: 

\begin{question}
	Can the idea of primary decomposition in Noetherian rings be extended to $S$-Noetherian rings? 
\end{question}

We provide a positive answer to the above question in this paper. A natural way 
to extend primary decomposition from Noetherian rings to the broader class of $S$-Noetherian rings is to replace "primary ideals" by "$S$-primary ideals" in the decomposition process. By doing this, we can logically adapt this powerful concept  that allows us to extend various  structural properties of Noetherian rings to  $S$-Noetherian rings.

In this paper, we introduce the concept of $S$-primary decomposition as a generalization of primary decomposition. We say that an ideal (disjoint from $S$) of a ring $R$ is $S$-primary decomposable or has an $S$-primary decomposition if it can be written as a finite intersection of $S$-primary ideals of $R$. First we  provide an example of an $S$-Noetherian ring in which primary decomposition does not exist (see Example \ref{non-laskerian}) which asserts that an $S$-Noetherian ring need not be a Laskerian ring in general. Then as one of our main results, we establish the existence of $S$-primary decomposition in $S$-Noetherian rings as a generalization of historical Lasker-Noether decomposition theorem (see Theorem \ref{sp} and Theorem \ref{on}). Among the other results,  we extend  first and second uniqueness theorems of primary decomposition to $S$-primary decomposition (see Theorem \ref{Q} and Theorem \ref{n}).
 
  Throughout the paper, $R$ will be a commutative ring with identity and $S$ be a multiplicative set of $R$ unless otherwise stated.

\section{$S$-Primary Decomposition in $S$-Noetherian Ring}

It is well known that  primary decomposition exists  in  a Noetherian ring. 
Recall from \cite{ah18} that a ring $R$ is said to have a \textit{Noetherian spectrum} if $R$ satisfies the ascending chain condition (ACC) on radical ideals. This is equivalent to the condition that $R$ satisfies the ACC on prime ideals, and each ideal has only finitely many prime ideals minimal over it. Also, a ring $R$  is said to be \textit{Laskerian} if each ideal of $R$ has a primary decomposition. 

Recall that \cite{ad02}, an ideal $I$ of $R$ is called $S$-finte if $sI\subseteq J\subseteq I$ for some finitely generated ideal $J$ of $R$ and some $s\in S$. Then $R$ is said to be an $S$-Noetherian ring if each ideal of $R$ is $S$-finite. We begin by providing an example of an $S$-Noetherian ring in which primary decomposition does not hold.

\begin{eg}\label{non-laskerian}
	Let  $R=F[x_1,x_2,\ldots, x_{n}, \ldots]$ be the polynomial ring  in infinitely many indeterminates over a field  $F$. Since $R$ has an ascending chain of prime ideals $(x_1)\subseteq (x_1,  x_2)\subseteq \cdots\subseteq  (x_1,x_2,\ldots, x_n)\subseteq\cdots$ which does not terminate, so $R$ has no Noetherian spectrum. This implies that $R$ is a non-Laskerian ring, by \cite[Theorem 4]{rg80}. Consider the multiplicative set $S=R\setminus\{0\}$. Then by  \cite[Proposition 2(a)]{ad02}, $R$ is an $S$-Noetherian ring. Hence $R$ is an $S$-Noetherian ring but not Laskerian.
\end{eg}

\noindent
Recall that let $f: R\longrightarrow S^{-1}R$ denote the usual homomorphism of rings given by $f(r)=\frac{r}{1}$. For any ideal $I$ of $R$, the  $f^{-1}(S^{-1}I)$ called the contraction of $I$ with respect to $S$, that is, $\{a\in R \hspace{0.2cm}| \hspace{0.2cm}\dfrac{a}{1}\in S^{-1}I\}$ is denoted by $S(I)$. Notice that $I\subseteq S(I)$. Thus we need an $S$-version of primary decomposition of ideals. Now we define the concept of $S$-primary decomposition of ideals as a generalization of primary decomposition.
\begin{deff}\label{dec.}
	\noindent
	Let $R$ be a ring and let $S$ be a multiplicative set of $R$. Let
	$I$ be an ideal of $R$ such that $I\cap S=\emptyset$. We say that $I$ admits an $S$-primary decomposition if $I$ is a finite intersection of $S$-primary ideals of $R$. In such a case, we say that $I$ is $S$-decomposable. An $S$-primary decomposition $I=\bigcap\limits_{i=1}^{n}Q_{i}$ of $I$ with $rad(Q_{i})=P_{i}$ for each $i\in\{1, 2,\ldots, n\}$ is said to be minimal if the following conditions hold:
	\leavevmode
	\begin{enumerate}
		\item $S(P_{i})\neq S(P_{j})$ for all distinct $i, j\in\{1, 2, \ldots, n\}$.
		\item $S(Q_{i})\nsupseteq\bigcap_{j\in\{1, 2, \ldots, n\}\setminus\{i\}}S(Q_{j})$ for each $i\in \{1,  2, \ldots, n\}$ (equivalently,\\
		$S(Q_{i})\nsupseteq \bigcap_{j\in\{1, 2, \ldots, n\}\setminus\{i\}}Q_{j}$ for each $i\in \{1,  2, \ldots, n\})$. 
	\end{enumerate}
\end{deff}

\noindent
 From definition \ref{dec.}, the concepts of $S$-primary decomposition and primary decomposition coincide for $S=\{1\}$. The following example shows that the concept of $S$-primary decomposition is a proper generalization of the concept of primary decomposition.

\begin{eg}
	\noindent
	Consider the Boolean ring $R=\prod_{n=1}^{\infty}\mathbb{Z}_{2}$ (countably infinite copies of  $\mathbb{Z}_{2})$. According to \cite[Theorem 1]{vd17}, the zero ideal $(0)=(0,0,0,\ldots)$ in $R$ does not have the primary decomposition. Clearly, $R$ is not Noetherian. Consider the multiplicative set $S=\{{1_R}=(1,1,1,\ldots),s=(1,0,0,\ldots)\}$. Let $I$ be an ideal of $R$. Either $s\in I$ or $s\notin I$. If $s\in I$, then $sI\subseteq Rs\subseteq I$ and as $Rs$ is a finitely generated ideal of $R$, it follows that $I$ is $S$-finite. If $s\notin I$, then $I\subseteq (0)\times\mathbb{Z}_{2}\times \mathbb{Z}_{2}\times\cdots$. In such a case, $sI\subseteq (0)\times (0)\times (0)\times\cdots\subseteq I$. Hence, $I$ is $S$-finite. This shows that  $R$ is an $S$-Noetherian ring. Next, we show that $(0)$ is an  $S$-primary ideal of $R$. First, we observe that $(0)\cap S=\emptyset$. Now, let $a=(a_n)_{n\in\mathbb{N}}$, $b=(b_n)_{n\in\mathbb{N}} \in R$ such that $ab=0$, where each $a_{i},  {b_{i}}\in \mathbb{Z}_{2}$. This implies that $a_nb_n=0$ for all $n\in \mathbb{N}$, in particular, $a_1b_1=0$. Then we have either $a_1=0$ or $b_1=0$. If $a_1=0$, then $sa=0$.  If $b_1=0$, then $sb=0$. Thus $(0)$ is an $S$-primary ideal. Therefore $(0)$ is $S$-primary decomposable.
\end{eg}

\noindent
Recall from \cite{fm69}, an ideal $I$ of the ring $R$ is called  \textit{irreducible} if $I = J \cap K$ for  some ideals $J$, $K$ of $R$, then either $I = J$ or $I = K$. It is well known that the classical proof of existence of primary decomposition in a Noetherian ring involves the concept of irreducible ideals. So we need $S$-version of irreducible ideals to prove the existence of $S$-primary decomposition in $S$-Noetherian rings.
 
\begin{deff}\label{S-iir.}
	An ideal $Q$ (disjoint from $S$) of the ring $R$ is called  $S$-irreducible if $s(I\cap J)\subseteq Q \subseteq I\cap J$ for some $s\in S$ and some ideals $I$, $J$ of $R$, then there exists $s'\in S$ such that either $ss'I\subseteq Q$ or $ss'J\subseteq Q$.
\end{deff}

\noindent
It is clear from the definition that every irreducible ideal is an $S$-irreducible ideal. However, the following example shows that an $S$-irreducibile ideal need not be irreducible.

\begin{eg}\label{fm}
	\noindent
	Let $R=\mathbb{Z}$, $S=\mathbb{Z}\setminus 3\mathbb{Z}$ and $I=6\mathbb{Z}$. Since $I=2\mathbb{Z}\cap 3\mathbb{Z}$, therefore $I$ is not an irreducible ideal of $R$. Now take $s=2\in S$. Then $2(3\mathbb{Z})=6\mathbb{Z}\subseteq I$. Thus $I$ is an $S$-irreducible ideal of $R$.
\end{eg}
\noindent
	Recall from \cite{ah20},  an ideal $Q$ (disjoint from $S$) of a ring $R$ is said to be  \textit{$S$-prime }if there exists an $s\in S$ such that for  $a,b\in R$ with $ab\in Q$, we have either $sa\in Q$ or $sb\in Q$.
Clearly, every $S$-prime ideal is $S$-primary.
The following example shows that the converse of this is not true in general.

\begin{eg}
	\noindent
	Consider  $R= \mathbb{Z}$, $Q=4\mathbb{Z}$, and  $S=\mathbb{Z}\setminus 2\mathbb{Z}$. Notice that   $4\in Q$ but  $2s\notin Q$ for all $s\in S$. This  implies that $Q$ is not an $S$-prime ideal of $R$. Obviously, $Q$ is a primary ideal of $R$ and hence $S$-primary.

\end{eg}	

	\noindent
	Recall from \cite[Proposition 2.5]{me22} that if $Q$ is an $S$-primary ideal of a ring $R$, then $P=\text{rad}(Q)$ is an $S$-prime ideal. In such a case, we say that $Q$ is an $P$-$S$-primary ideal of $R$.

\begin{prop}\label{ry}
	Let $S$ be a multiplicative set of a ring $R$. Then the following statements hold:
	\begin{enumerate}
		\item Finite intersection of $P$-$S$-primary ideals is $P$-$S$-primary.
		\item If $Q$ is a $P$-primary ideal of $R$ with $Q\cap S=\emptyset$, then for any ideal $J$ of $R$ with $J\cap S\neq\emptyset$, $Q \cap J$ is a $(P\cap rad(J))$-$S$-primary ideal of $R$.
	\end{enumerate}
\end{prop}
\begin{proof}
	\leavevmode
	\begin{enumerate}
		\item Let $Q_1,Q_2,\ldots,Q_n$ be $P$-$S$-primary ideals, then $S\cap Q_{i} =\emptyset$ for each $i=1,2,\ldots,n$, and so  $S\cap(\bigcap\limits_{i=1}^{n}Q_{i}) =\emptyset$. Suppose $Q=\bigcap\limits_{i=1}^{n}Q_{i}$. Since each $Q_{i}$ is $P$-$S$-primary, $rad(Q)=rad(\bigcap\limits_{i=1}^{n}Q_{i})=\bigcap\limits_{i=1}^{n}rad(Q_{i})=P$. Now, let $xy\in Q$,  where $x,y\in R$ and with $sy\notin Q$ for all $s\in S$. Consequently, for every $s\in S$, there exists $k_{s}$ such that $xy\in Q_{k_{s}}$ and $sy\notin Q_{k_{s}}$. Let $s_{i}\in S$ be the element satisfying the $S$-primary property for $Q_{i}$. Since we have finitely many $Q_{i}$, put $s=s_{1}s_{2}\dots s_{n}\in S$. Now fix $s$ and assume that $xy\in Q$ but $sy\notin Q$. Thus, there exists $k$ such that $xy\in Q_{k}$ and $sy\notin Q_{k}$. Then for $s_{k}\in S$ we obtain $s_{k}x\in rad(Q_k)=P$ or $s_{k}y\in Q_{k}$. The latter case gives $sy\in Q_{k}$, a contradiction. Thus $sx\in rad(Q)=P$, and therefore $Q$ is $P$-$S$-primary.
		\item As $Q\cap S=\emptyset$, it follows that $(Q\cap J)\cap S=\emptyset$. By assumption, $J\cap S\neq\emptyset$.
		Let $s \in J\cap S$. Let $a, b\in R$ be such that $ab\in Q \cap J$. Either $a\in Q$	or $b\in rad(Q) = P$, since $Q$ is a $P$-primary ideal of $R$. Hence, either $sa\in Q \cap J$ or $sb \in P\cap J \subseteq P\cap rad(J) = rad(Q)\cap rad(J) = rad(Q \cap J)$. This proves that $Q\cap J$ is a $(P\cap rad(J))$-$S$-primary ideal of $R$.
	\end{enumerate}
\end{proof}

\noindent
If an ideal $I$ of a ring $R$ admits an $S$-primary decomposition, then we prove in Remark \ref{11} that it will admit a minimal $S$-primary decomposition.\\

\noindent
    Following \cite{sh16}, 
	 let  $E$ be a family of ideals of a ring $R$. An element $I\in E$ is said to be an \textit{$S$-maximal element} of $E$ if there exists an $s\in S$ such that for each $J\in E$, if $I\subseteq J$, then $sJ\subseteq I$. Also an increasing sequence $(I_{i})_{i\in\mathbb{N}}$ of ideals of $R$ is called \textit{$S$-stationary} if there exist a positive integer $k$ and $s\in S$ such that  $sI_{n}\subseteq I_{k}$ for all $n\geq k$.\\
		
\noindent	
The following theorem provides a connection between the concepts of $S$-irreducible ideals and $S$-primary ideals.

\begin{thm}\label{sp}
	\noindent
	Let R be an $S$-Noetherian ring. Then every $S$-irreducible ideal of $R$ is $S$-primary.
\end{thm}
\begin{proof}
	Suppose $Q$ is an $S$-irreducible ideal of $R$. Let $a,b\in R$ such that $ab\in Q$ and $sb\notin Q$ for all $s\in S$. Our aim is to show that there exists $t\in S$ such that $ta\in rad(Q)$.
 Consider  $A_n=\{x\in R\hspace{0.2cm}|\hspace{0.1cm} a^{n}x\in Q \}$ for $n\in\mathbb{N}.$ 
	Clearly each $A_n$ is an ideal of $R$ and $A_1\subseteq A_2\subseteq A_3\subseteq \cdots$ is an increasing chain of ideals of $R$. Since $R$ is $S$-Noetherian, by \cite[Theorem 2.3]{zb17}, this chain is  $S$-stationary, i.e., there exist $k\in \mathbb{N}$ and $s\in S$ such that $sA_n\subseteq A_k$ for all $n\geq k$. Consider the two ideals $I=\textless a^{k}\textgreater +\hspace{0.1cm} Q$ and $J=\textless b\textgreater +\hspace{0.1cm} Q$ of $R$. Then  $Q\subseteq I\cap J$. For the reverse containment, let $y\in I\cap J$. Write  $y=a^{k}z+q$ for some $z\in R$ and $q\in Q$. Since $ab\in Q$, $aJ\subseteq Q$; whence $ay\in Q$. Now $a^{k+1}z=a(a^{k}z)=a(y-q)\in Q$. This implies that $z\in A_{k+1}$, and so  $sz\in sA_{k+1}\subseteq A_k$. Consequently, $a^{k}sz\in Q $ which implies that  $a^{k}sz +sq=sy\in Q$. Thus we have $s(I\cap J)\subseteq Q\subseteq I\cap J$. This implies that there exists $s'\in S$ such that either $ss'I\subseteq Q$ or $ss'J\subseteq Q$ since $Q$ is  $S$-irreducible. If $ss'J\subseteq Q$, then $ss'b\in Q$  which is not possible. Therefore $ss'I\subseteq Q$ which implies that $ss'a^{k}\in Q$. Put $t=ss'\in S$. Then $(ta)^{k} \in Q$ and hence  $ta\in rad(Q)$, as desired.
\end{proof}

\noindent
Now we are in a position to prove the existence of $S$-primary decomposition in $S$-Noetherian rings as our main result.
	
\begin{thm}(\textbf{Existence of $S$-Primary Decomposition})\label{on}
	\noindent
	Let $R$ be an $S$-Noetherian ring. Then every proper ideal of $R$ disjoint from $S$ can be written as a finite intersection of $S$-primary ideals.
\end{thm}

\begin{proof}
	\noindent
	Let $E$ be the collection of ideals of $R$ which are disjoint from $S$ and can not be written as a finite intersection of $S$-primary ideals. We wish to show $E=\emptyset$. On the contrary suppose  $E\neq\emptyset$.  Since $R$ is an  $S$-Noetherian ring, by \cite[Theorem 2.3]{zb17}, there exists an $S$-maximal element in $E$, say $I$. Evidently, $I$ is not an $S$-primary ideal, by Theorem \ref{sp}, $I$ is not an $S$-irreducible ideal, and so  $I$ is not an irreducible ideal. This implies that  $I=J\cap K$ for some ideals $J$ and $K$ of $R$ with $I\neq J$ and $I\neq K$. Since $I$ is not $S$-irreducible, $sJ\nsubseteq I$ and $sK\nsubseteq I$ for all $s\in S$. Now we claim that $J$, $K\notin E$. For this, 
	if $J$ (respectively, $K$) belongs to $E$, then since $I$ is an $S$-maximal element of $E$ and $I\subset J$ (respectively, $I\subset K$), there exists $s'$ (respectively, $s''$) from $S$ such that $s'J \subseteq I$ (respectively, $s''K\subseteq I$). This is not possible, as $I$ is not $S$-irreducible. Therefore $J, K\notin E$. This implies that $J$ and $K$ can be written as a finite intersection of $S$-primary ideals. Consequently, $I$ can also be written as a finite intersection of $S$-primary ideals since  $I=J\cap K$, a contradiction as $I\in E$. Thus $E=\emptyset$, i.e., every proper ideal of $R$ disjoint from $S$ can be written as a finite intersection of $S$-primary ideals. 
	\end{proof}

\noindent
Recall that if $I$ is any ideal of $R$, the radical of $I$ is $rad(I)=\{x\in R \hspace{0.1cm}| \hspace{0.1cm}x^{n}\in I \hspace{0.1cm}for\hspace{0.1cm} some \hspace{0.1cm}n\textgreater 0\}$. If $I=rad(I)$, then $I$ is called a radical ideal.

\begin{cor}\label{rad ideal}
	\noindent
	Let $R$ be an $S$-Noetherian ring. Then every radical ideal $I$ (disjoint from $S$) of $R$ is the intersection of finitely many $S$-prime ideals.
\end{cor}

\begin{proof}
	By Theorem \ref{on}, there exist finitely many $S$-primary ideals $Q_1,Q_2,\ldots,Q_r$ such that $I=Q_1\cap Q_2\cap\cdots\cap Q_r$. Also by \cite[Proposition 2.5]{me22}, $rad ({Q_i})=P_i$ is $S$-prime for $i=1,2,\ldots,r$. Consequently, $I=rad(I)=P_1\cap P_2\cap\cdots\cap P_r$.
	This completes the proof.
\end{proof}
\begin{rem}\label{11}
	\noindent
	Let $R$ be a ring, and let $S$ be a multiplicative set of $R$. Let $I$ be an ideal of $R$ such that $I\cap S=\emptyset$. Suppose that $I$ admits an $S$-primary decomposition. (In Theorem \ref{on}, we have shown that if $R$ is $S$-Noetherian, then any ideal $I$ of $R$ disjoint from $S$ admits an $S$-primary decomposition). Let $I=\bigcap\limits_{i=1}^{n}Q_{i}$ be an $S$-primary decomposition of $I$ with $Q_{i}$ is a $P_{i}$-$S$-primary ideal of $R$ for each $i\in\{1,2,\ldots, n\}$. Then, by \cite[Proposition 2.7]{me22}, $S^{-1}Q_{i}$ is $S^{-1}P_{i}$-primary, $S(I)=\bigcap\limits_{i=1}^{n}S(Q_{i})$ is a primary decomposition with $S(Q_{i})$ is a $S(P_{i})$-primary for each $i\in\{1,\ldots, n\}$. Let $1\leq i\leq n$. Let $s_{i}\in S$ be such that it satisfies the $S$-primary property of $Q_{i}$. Notice that $(P_{i}:s_{i})$ is a prime ideal of $R$ and $(Q_{i}:s_{i})$ is a $(P_{i}:s_{i})$-primary ideal of $R$. Let $s=\prod_{i=1}^{n}s_{i}$. Then $s\in S$. Observe that $(Q_{i}:s_{i})=(Q_{i}:s)$, $(P_{i}:s_{i})=(P_{i}:s)$, $S(Q_{i})=(Q_{i}:s)$, and $S(P_{i})=(P_{i}:s)$. From $I=\bigcap\limits_{i=1}^{n}Q_{i}$, it follows that $(I:s)=\bigcap\limits_{i=1}^{n}(Q_{i}:s)=\bigcap\limits_{i=1}^{n}S(Q_{i})=S(I)$. Let $k$ of the $S(P_1), \ldots, S(P_n)$ be distinct. After a suitable rearrangement of $\{1, \ldots, n\}$, we can assume without loss of generality that $S(P_{1}),\ldots, S(P_k)$ are distinct among $S(P_{1}),\ldots, S(P_n)$. Let $A_{1}=\{j\in\{1, \ldots, n\}\mid S(Q_{j})~ is ~S(P_{1})$-primary$\}$, \ldots, $A_{k}=\{j\in\{1, \ldots, n\}\mid S(Q_{j})~ is ~S(P_{k})$-primary $\}$. 
	
	From the above discussion, it is evident that $1 \in A_1,\ldots, k\in A_{k}$, and $\{1,\ldots,n\}=\bigcup\limits_{t=1}^{k}A_{t}$. Let $1\leq t\leq k$. Notice that $\bigcap_{j\in A_{t}}S(Q_{j})$ is $S(P_{t})$-primary by \cite[Lemma 4.3]{fm69}. It is convenient to denote $\bigcap_{j\in A_{t}}Q_{j}$ by $I'_{t}$. Thus $(I:s)=S(I)=S(I'_{1})\cap S(I'_{2})\cap \cdots\cap S(I'_{k})$ with $S(I'_{t})$ is $S(P_{t})$-primary for each $t\in \{1,\ldots, k\}$ and $S(P_1),\ldots, S(P_k)$ are distinct. After omitting those $S(I'_{i})$ such that $S(I'_{i})\supseteq \bigcap_{t\in\{1,\ldots, k\}\setminus\{i\}}S(I'_{t})$ from the intersection, we can assume  without loss of generality that $(I:s)=S(I)=\bigcap\limits_{t=1}^{k}S(I'_{t})$ is a minimal primary decomposition of $S(I)$. Next, we claim that $I=(I:s)\cap (I+Rs)$. It is clear that $I\subseteq (I:s)\cap (I+Rs)$. Let $y\in(I:s)\cap (I+Rs)$. Then $ys\in I$ and $y=a+rs$ for some $r\in R$. This implies that $ys=as+rs^{2}$ and so, $rs^{2}\in I$. Hence, $r\in S(I)=(I:s)$. Therefore, $y=a+rs\in I$. This shows that $(I:s)\cap (I+Rs)\subseteq I$. Thus $I=(I:s)\cap (I+Rs)$ and hence, $I=(\bigcap\limits_{t=1}^{k}S(I'_{t})\cap (I+Rs))=\bigcap\limits_{t=1}^{k}(S(I'_{t})\cap (I+Rs))$. Let $1\leq t\leq k$.  For convenience, let us denote $S(I'_{t})\cap (I+Rs)$ by $Q'_{t}$. As $S(I'_{t})$ is $S(P_{t})$-primary with $S(I'_{t})\cap S=\emptyset$ and $(I+Rs)\cap S\neq \emptyset$, we obtain from Proposition \ref{ry}(2) that $Q'_{t}$ is $S(P_{t})\cap rad(I+Rs)$-$S$-primary. Since $S(S(P_{t}))=S(P_{t})$, $S(rad(I+Rs))=R$, it follows that $S(S(P_{t})\cap rad(I+Rs))=S(P_{t})$. Notice that $S(Q'_{t})=S(S(I'_{t})\cap (I+Rs))=S(I'_{t})$, as $S(S(I'_{t}))=S(I'_{t})$ and $S(I+Rs)=R$. Hence, for all distinct $i, j\in\{1,\ldots, k\}$, $S(S(P_{i})\cap rad(I+Rs))\neq S(S(P_{j})\cap rad(I+Rs))$ and for each $i$ with $1\leq i\leq t$, $S(Q'_{i})\nsupseteq \bigcap_{t\in\{1,\ldots, k\}\setminus\{i\}}S(Q'_{t})$. Therefore, $I=\bigcap\limits_{t=1}^{k}Q'_{t}$ is a minimal $S$-primary decomposition.

\end{rem}

\noindent
Recall from \cite{ah20}, let $R$ be a ring, $S\subseteq R$ a multiplicative set and $I$ an ideal of $R$ disjoint from $S$. Let $P$ be an $S$-prime ideal of $R$ such that $I\subseteq P$. Then $P$ is said to be a \textit{minimal $S$-prime ideal } over $I$ if $P$ is minimal in the set of the $S$-prime ideals containing $I$. Also Ahmed \cite[Remark 2]{ah20} proves in $S$-Noetherian rings that the set of minimal $S$-prime ideals is finite if $S$ is a finite multiplicative set.

\noindent
A prime ideal is said to be a \textit{minimal prime ideal} if it is a minimal prime ideal over the zero ideal. Emmy Noether showed that in a Noetherian ring, there are only finitely many minimal prime ideals over any given ideal \cite[ Theorem 88]{ki74}. A natural question arises: 
\begin{question}\label{tion}
	Is the collection of minimal prime ideals in $S$-Noetherian rings finite for any multiplicative set $S$?
\end{question}
The answer to the above question is negative. In the following example, we provide an $S$-Noetherian ring that has infinitely many minimal prime ideals.

\begin{eg}
	Consider the ring $R=\dfrac{F[x_1,x_2, \ldots, x_{n}, \ldots]}{(x_{i}x_{j}; i\neq j,i,j\in\mathbb{N})}$, where $F$ is a field. Let $y_{i}={\overline{x}}_i$ be the image of $x_i$ under the canonical map. Consider the  multiplicative set $S=\{y_1^{n}: n\in\mathbb{N}\cup \{0\}$\} of $R$. Then $(y_1)\subseteq (y_1,  y_2)\subseteq\cdots \subseteq (y_1, y_2,\ldots, y_n) \subseteq \cdots$ is an ascending chain of ideals of $R$ which is not stationary. Consequently, $R$ is not a Noetherian ring. Evidently,  $R$ is an $S$-Noetherian ring  (see \cite[Example 4]{ts23}). Note that if $P\subset R$ is a prime ideal, then there can be at most one $i$ such that $y_i\notin P$. It follows that if $P_i$ is the ideal generated by all the $y_j$ for $j\neq i$, each $P_i$ is prime since $R/P_i\cong F[x_i]$. Clearly, every $P_{i}$ is a minimal prime ideal of $R$. Thus $R$ has infinitely many minimal prime ideals.
\end{eg}
\begin{question}
	Under what condition the set of minimal prime ideals is finite in $S$-Noetherian rings?
\end{question}
\noindent
We give an answer to the above question under a mild condition on the multiplicative set $S\subseteq R$ (see Theorem \ref{min}).

\begin{rem}\label{ps}
Let $R$ be a ring, $S\subseteq R$ be a multiplicative set, and $(0)=Q_1\cap Q_2\cap\cdots\cap Q_r$ be a minimal $S$-primary decomposition, where $Q_{i}$ is $P_{i}$-$S$-primary ideal of $R$. If $Z(R/P_{i})\cap \overline{S}_{i}=\emptyset$,  where $\overline{S}_{i}=\{s+P_{i}\mid s\in S\}$ and $Z(R/P_{i})$ denotes the set of zero divisors of $R/P_{i}$ for all $i=1, 2, \ldots, r$. Then each $P_{i}$ is a prime ideal of $R$. For this, we will show that $(P_{i} : s) = P_{i}$ for all $s\in S$ and $i=1, 2, \ldots, r$. It is clear that $P_{i}\subseteq (P_{i}: s)$. Conversely, let $s\in S$ and $x\in (P_{i}: s)$, then $sx\in P_{i}$; so $(s+P_{i})(x+P_{i})= P_{i}$. Thus $x\in P_{i}$ since $Z(R/P_{i})\cap\overline{S}_{i} = \emptyset$. Now, if $P_{i}$ is a $S$-prime ideal of $R$, then by \cite[Proposition 1]{ah20}, $(P_{i}: s_{i})$ is a prime ideal of $R$ for some $s_{i}\in S$. Since $P_{i}= (P_{i}: s_{i})$, $P_{i}$ is a prime ideal of $R$.
\end{rem}

\begin{theorem}\label{min}
Let $R$ be an $S$-Noetherian ring, $S\subseteq R$ a multiplicative set, and $(0)=Q_1\cap Q_2\cap\cdots\cap Q_r$ be a minimal $S$-primary decomposition, where $Q_{i}$ is $P_{i}$-$S$-primary ideal of $R$. If $Z(R/P_{i})\cap \overline{S}_{i}=\emptyset$, then minimal prime ideals of $R$ are in the set $\{P_{i}\mid i=1, 2,\ldots r\}$.
\end{theorem}
\begin{proof}
 Since $(0)=Q_1\cap Q_2\cap\cdots\cap Q_r$, $rad~(0)=P_1\cap P_2\cap\cdots\cap P_r$, where $P_{i}=rad(Q_{i})$ $(1\leq i\leq r)$ is an $S$-prime ideal of $R$  By Remark \ref{ps}, each $P_i$ $(1\leq i\leq r)$ is a prime ideal of $R$. Let $P'$ be any minimal prime ideal of  $R$. Then $(0)\subseteq P'$, and so $rad~(0)\subseteq rad(P')=P'$. This implies that $P_1\cap P_2\cap\cdots\cap P_r \subseteq P'$.  By \cite[Proposition 1.11(ii)]{fm69}, there exists $j\in \{1,2,\ldots,r\}$ such that $P_{j}\subseteq P'$. Since $P'$ is a minimal prime ideal, $P'=P_{j}$, as desired. 
\end{proof}

\begin{prop}\label{k}
	\noindent
	Let $Q_{1},Q_{2} ,\ldots,Q_{n}$ be ideals of $R$, and  $P$ be an $S$-prime ideal containing $\bigcap\limits_{k=1}^{n}Q_{k}$. Then there exists $s\in S$ such that $sQ_k\subseteq P$ for some $k$. In particular, if  $P=\bigcap\limits_{k=1}^{n}Q_{k}$, then $sQ_{k}\subseteq P\subseteq Q_{k}$ for some $k$.
\end{prop}
\begin{proof}
	\noindent
	On contrary suppose $sQ_{k}\nsubseteq P$ for all $s\in S$ and $k= 1,2,\dots, n$. Then for each $k$, there exists $x_{k}\in Q_k$ such that $sx_{k}\notin P$ for all $s\in S$. Evidently, $x_1x_2\ldots x_n\in Q_1Q_2\ldots Q_n$. This implies that $x_1x_2\cdots x_n\in Q_1Q_2\ldots Q_n\subseteq \bigcap\limits_{k=1}^{n}Q_{k}\subseteq P$. Since $P$ is an $S$-prime ideal, by \cite[Proposition 4]{ah20}, there exist $s'\in S$ and $j\in\{1,2,\ldots, n\}$ such that $s'x_{j}\in P$, a contradiction. Hence there exist $s\in S$ and $k\in\{1,2,\ldots,n\}$ such that $sQ_{k}\subseteq P$. Finally, if $P=\bigcap\limits_{k=1}^{n}Q_{k}$, then by above argument there exist $s\in S$ and $k\in\{1,2,\ldots, n\}$ such that $sQ_{k}\subseteq P\subseteq Q_{k}$.
\end{proof}

\noindent
Now we prove $S$-version of the $1st$ uniqueness theorem:
\begin{thm}$(\textbf{$S$-version of 1st $S$-uniqueness theorem})$.\label{Q}
	\noindent
	Let $R$ be a ring and let $S$ be a multiplicative set of $R$. Let $I$ be an ideal of $R$ which admits $S$-primary decomposition. Let $I=\bigcap\limits_{i=1 }^{n}Q_i$ be a minimal $S$-primary decomposition, where $Q_i$ is $P_i$-$S$-primary for each $i\in \{1,  2, \ldots, n\}$. Then the $S(P_i)$ are precisely the prime ideals which occur in the set of ideals $S(rad(I : x))$ $(x\in R)$, and hence are independent of the particular $S$-decomposition of $I$.
\end{thm}
	\begin{proof}
		\noindent
	For any $x\in R$, we have $S^{-1}(I:x)=(\bigcap\limits_{i=1}^{n}S^{-1}Q_{i}:_{S^{-1}R}\frac{x}{1})=\bigcap\limits_{i=1}^{n}(S^{-1}Q_{i}:_{S^{-1}R}\frac{x}{1})$, and so  $S(I:x)=\bigcap\limits_{i=1}^{n}(S(Q_{i}):x)$. Evidently,  $S(rad(I:x))=\bigcap\limits_{i=1}^{n}rad(S(Q_{i}):x)$. Notice that each $S(P_{i})$ is a prime ideal since $S^{-1}P_{i}$ is a prime ideal and contraction of a prime ideal is a prime ideal. Now we show that each $S(Q_{i})$ is $S(P_{i})$-primary. For this, let $x, ~y \in R$ such that $xy\in S(Q_{i})$. Then $\dfrac{xy}{1}\in S^{-1}Q_{i}$, $\dfrac{xy}{1}=\dfrac{a}{s}$ for some $a\in Q_{i}$ and $s\in S$. Consequently, there exists $u\in S$ such that $usxy\in Q_{i}$. Since $Q_{i}$ is $P_{i}$-$S$-primary, there exists $t\in S$ such that either $tusx\in Q_{i}$ or $ty\in P_{i}$. It follows that either $\dfrac{x}{1}=\dfrac{tusx}{tus}\in S^{-1}Q_{i}$ or $\dfrac{y}{1}=\dfrac{ty}{t}\in S^{-1}P_{i}$. This implies that either $x\in S(Q_{i})$ or $y\in S(P_{i})$, and  rad~$S(Q_{i})=S(rad(Q_{i}))=S(P_{i})$. Therefore $S(Q_{i})$ is $S(P_{i})$-primary. Then $S(rad(I:x))=\bigcap\limits_{i=1}^{n}rad(S(Q_{i}):x)=  \bigcap\limits_{x\notin S(Q_{i})}S(P_{i})$, by \cite[Lemma 4.4]{fm69}. Suppose $S(rad(I:x))$ is prime; then we have $S(rad(I:x))=S(P_{j})$ for some $j$, by \cite[Proposition 1.11]{fm69}. Hence every prime ideal of the form $S(rad(I:x))$ is one of the $S(P_{j})$. Conversely, for each $i$ there exists $x_{i} \in \left(\bigcap_{j\in\{1, 2, \ldots, n\}\setminus\{i\}}S(Q_{j})\right)\setminus S(Q_{i})$ since the $S$-decomposition is minimal. Consequently, for each $i$, we have $S(I:x_{i})=\bigcap\limits_{i=1}^{n}(S(Q_{i}):x_{i})=(S(Q_{i}):x_{i})$ since $(S(Q_j):x_i) =R$ for all $j\neq i$, by \cite[Lemma 4.4]{fm69}. Then $S(rad(I:x_{i}))=S(P_{i})$, by \cite[Lemma 4.4]{fm69}. This completes the proof. 
\end{proof}

\noindent
To prove the second uniqueness theorem for the $S$-Noetherian rings, we need the following results:

\begin{lem}\label{R}
	\noindent
	Let $R$ be a ring, $S\subseteq R$ be  a multiplicative set, and  $Q$ be an $P$-$S$-primary ideal of $R$. Then $S^{-1}Q$ is $S^{-1}P$-primary and its contraction in $R$ satisfies $tS(Q)\subseteq Q\subseteq S(Q)$ for some $t\in S$. 
\end{lem}

\begin{proof}
\noindent
 Clearly, $S^{-1}Q\neq S^{-1}R$ since $Q\cap S=\emptyset$ and $P=rad(Q)$. Let $x\in S(Q)$. Then  $\dfrac{x}{1}=\dfrac{a}{s}$ for some $a\in Q$ and $s\in S$. Consequently, there exists $s'\in S$ such that $(xs-a)s'=0$; whence  $ss'x=s'a\in Q$. As $Q$ is $S$-primary, there exists $t\in S$ such that either $ss't\in rad(Q)=P$ or $tx\in Q$. This implies that $tx\in Q$ since $P\cap S=\emptyset$, and so  $tS(Q)\subseteq Q\subseteq S(Q)$. Now since $Q$ is $S$-primary, by \cite[Proposition 2.7]{me22}, $S^{-1}Q$ is a primary ideal of $S^{-1}R$  and  $rad(S^{-1}Q)=S^{-1}(rad(Q))=S^{-1}P$. Also since $P$ is an $S$-prime ideal of $R$, by \cite[Remark $1$]{ah20}, $S^{-1}P$ is a prime ideal of $S^{-1}R$. Thus $S^{-1}Q$ is an $S^{-1}P$-primary ideal of $S^{-1}R$. 
\end{proof}

\begin{lem}\label{s}
	\noindent
	Let $R$ be a ring and let $S$ be a multiplicative set of $R$. Let $I$
	be an ideal of $R$ such that $I$ admits an $S$-primary decomposition. Let $I=\bigcap\limits_{i=1}^{n}Q_i$ be a minimal $S$-primary decomposition where $Q_i$ is $P_i$-$S$-primary for each $i\in\{1, 2,\ldots , n\}$. Let $S'$
	be a multiplicative set of $R$ such that $S\subseteq S'$. If $Q_1, Q_{2}, \ldots, Q_m$ are such that $Q_i \cap S'=\emptyset$ for each $i\in \{1,\ldots, m\}$ and $Q_i\cap S'\neq\emptyset$ for each $i\in \{m + 1, \ldots, n\}$, then $S'^{-1}I=\bigcap\limits_{i=1}^{m}S'^{-1}Q_{i}$ is a primary decomposition and there exists $t\in S$ such that $tS'(I)\subseteq\bigcap\limits_{i=1}^{m}Q_{i}\subseteq S'(I)$.
	
\end{lem}
\begin{proof}
\noindent
By \cite[Proposition 3.11(v)]{fm69}, $S'^{-1}I=\bigcap\limits_{i=1}^{n}S'^{-1}Q_{i}$. Since $S'\cap Q_i\neq \emptyset$ for $i=m+1,m+2,\ldots, n$,  by \cite[Proposition 4.8(i)]{fm69}, $S'^{-1}Q_{i}=S'^{-1}R$ for $i=m+1,m+2,\ldots, n$. Consequently,  
$S'^{-1}I=\bigcap\limits_{i=1}^{m}S'^{-1}Q_{i}$. As $Q_{i}$ is $P_{i}$-$S$-primary and $S\subseteq S'$, it follows that $Q_{i}$ is $P_{i}$-$S'$-primary for each $i=1,\ldots, m$. Also by Lemma \ref{R}, $S'^{-1}Q_{i}$ is $S'^{-1}P_{i}$-primary for $i=1,2,\ldots, m$. Thus  $S'^{-1}I=\bigcap\limits_{i=1}^{m}S'^{-1}Q_{i}$ is a primary decomposition. Next, let $x\in S'(I)$. Then $\dfrac{x}{1}\in S'^{-1}I=\bigcap\limits_{i=1}^{m}S'^{-1}Q_{i}$. Write $\dfrac{x}{1}= \dfrac{a_{i}}{s_{i}}$ for some $a_{i}\in Q_{i}$ and $s_{i}\in S'$. This implies that $(xs_{i}-a_{i})s=0$ for some $s\in S'$; whence $s_{i}sx=sa_{i}\in Q_{i}$ for $i=1,2,\ldots,m$. Then there exists $t_{i}\in S$ such that $t_{i} x\in Q_{i}$ since $Q_{i}$ is $P_{i}$-$S$-primary and $P_{i}\cap S'=\emptyset$ for all $i=1, 2,\ldots, m$. Put $t=t_1t_2\ldots t_m$. Then $tx\in\bigcap\limits_{i=1}^{m}Q_{i}$, and hence $tS'(I)\subseteq \bigcap\limits_{i=1}^{m}Q_{i}\subseteq S'(\bigcap\limits_{i=1}^{m}Q_{i})=\bigcap\limits_{i=1}^{m}S'(Q_{i})=S'(I)$. Hence $tS'(I)\subseteq \bigcap\limits_{i=1}^{m}Q_{i}\subseteq S'(I)$.
\end{proof} 

	\noindent
	For an ideal $I$ of $R$, let $I=\bigcap\limits_{i=1}^{n}Q_{i}$ be a minimal $S$-primary decomposition, where $Q_{i}$ is $P_i$-$S$-primary. Then the $S$-prime ideals $P_{i}$ are said to belong to $I$. Also we say that $P_{i}$ ($1\leq i\leq n$) is \textit{isolated $S$-prime} if $sP_{j}\nsubseteq P_{i}$ for all $s\in S$ and for all $j\neq i$. Other $S$-prime ideals are called embedded $S$-prime ideals.

\begin{thm}(\textbf{$S$-version of 2nd uniqueness theorem}).\label{n}	
	\noindent
	Let $I$ be an ideal of $R$ which admits $S$-primary decomposition, that is, $I=\bigcap\limits_{i=1}^{n}Q_{i}$ be a minimal $S$-primary decomposition of $I$, where $Q_{i}$ is $P_i$-$S$-primary. If  $\{P_{{1}},\ldots,P_{{m}}\}$ is a set of isolated $S$-prime ideals of $I$ for some $m$ $(1\leq m\leq n)$, then $Q_{{1}}\cap \cdots\cap Q_{{m}}$ is independent  of  $S$-primary decomposition.
\end{thm}

\begin{proof}
	Since each $P_{i}$ is $S$-prime, by \cite[Proposition 1]{ah20},  there exists $s_{i}\in S$ such that $(P_{i}: s_i)$ is a prime ideal for  $i=1,2,\ldots, m$. Consider the multiplicative set $S'=R-\bigcup\limits_{i=1 }^{m}(P_{i}:s_{i})$ of $R$. Evidently, $S'\cap P_{i} =\emptyset$ for $i=1,2,\ldots m$ since $P_{i}\subseteq (P_{i}:s_{i})$. We claim that $S'\cap P_{k} \neq \emptyset$ for $k=m+1,m+2,\ldots, n$. On the contrary $ S'\cap P_{k} =\emptyset$ for some $k$. This implies that $P_{k}\subseteq \bigcup\limits_{i=1 }^{m}(P_{i}:s_{i})$,  by \cite[Proposition 1.11(i)]{fm69}, $P_k\subseteq (P_j:s_j)$ for some $j$ $(1\leq j\leq m)$. Consequently, $s_jP_k\subseteq P_j$, a contradiction since $P_j$ is isolated $S$-prime. This concludes that  $S'\cap\left(\bigcup\limits_{i=1}^{m}Q_{i}\right)=\emptyset$ and $S'\cap\left(\bigcap\limits_{i=m+1}^{n}Q_{i}\right)\neq\emptyset$ since $P_i=rad(Q_i)$ for all $i=1,2,\ldots,n.$ Now since each $P_{i}$ is $S$-prime,  $S\cap P_{i}=\emptyset$ for all $i=1,2,\ldots,m$. This implies that $S\cap(P_i:s_i)=\emptyset$ for all $i=1,2,\ldots,m$. For, if $S\cap(P_i:s_i)\neq\emptyset$ for some $i$, then there exists $s'\in S$ such that $s's_i \in P_i $, a contradiction. Consequently,   $S\cap(\bigcup\limits_{i=1 }^{m}(P_{i}:s_{i}))=\emptyset$, and so $S\subseteq R-\bigcup\limits_{i=1 }^{m}(P_{i}:s_{i})=S'$. Also since each $Q_{i}$ is $S$-primary and  $S'\cap Q_{i}=\emptyset$ for all $i$ $(1\leq i\leq m)$, $Q_{i}$ is $S'$-primary for all $i=1,2,\ldots,m.$ Thus by Lemma \ref{s}, we have   $tS'(I)\subseteq\bigcap\limits_{j=1}^{m}Q_{j}\subseteq S'(I)$ for some $t\in S$. From above it is clear that $Q_{{1}}\cap \cdots\cap Q_{{m}}$ depends only $I$, and hence $Q_{{1}}\cap \cdots\cap Q_{{m}}$ is independent of $S$-primary decomposition.
\end{proof} 
\noindent \textbf{Acknowledgement:}  Authors sincerely thank the anonymous referee for  thorough review and very useful comments and suggestions that helped to improve this paper.


\begin{thebibliography}{999}
	\bibitem{ah18} 
	H. Ahmed :\textit{$S$-Noetherian spectrum condition.} \textit{Comm. Algebra.} \textbf{46}(2018), 3314-3321. Zbl 1395.13016. MR3788995. https://doi.org/10.1080/00927872.2017.1412455
	
	\bibitem{ah20} 
	H. Ahmed, M. Achraf :\textit{On $S$-prime ideals in commutative rings.} \textit{Comm. Algebra.}  \textbf{48}(2020), 4263-4273. Zbl 1455.13011. MR4145364. https://doi.org/10.1007/s13366-019-00476-5
 \label{key}
		\bibitem{sh16} 
	H. Ahmed, H. Sana :\textit{Modules satisfying the $S$-Noetherian property and $S$-ACCR}.  Comm. Algebra. \textbf{44}(2016), 1941-1951. Zbl 1347.13005. MR3490657. https://doi.org/10.1080/00927872.2015.1027377
	
	\bibitem{sh15} 
	H. Ahmed, H. Sana :\textit{$S$-Noetherian rings of the Forms $\# [X]$ and $\#[[X]]$}. Comm. Algebra. \textbf{43}(2015), 3848-3856. Zbl 1329.13014. MR3360852.  https://doi.org/10.1080/00927872.2014.924127
	
	\bibitem{ad02} 
	D. D. Anderson, T. Dumitrescu :\textit{$S$-Noetherian rings.}  Comm. Algebra. \textbf{30}(2002), 4407-4416. Zbl 1060.13007. MR1936480. https://doi.org/10.1081/AGB-120013328


	\bibitem{fm69} 
	M. F. Atiyah, I. G. MacDonald :\textit{Introduction to Commutative Algebra}. Addison-Wesley Publishing Company. (1969).  Zbl 0175.03601. MR0242802. 
	
	\bibitem{bj16} 
	J. Baeck, G. Lee, J. Wook Lim :\textit{$S$-Noetherian Rings and Their Extensions.} Taiwanese J. Math. \textbf{20}(2016), 1231-1250. Zbl 1357.16039. MR3580293.  https://doi.org/10.11650/tjm.20.2016.7436
	
	\bibitem{zb17}
	Z. Bilgin, M.L. Reyes, U. Tekir :\textit{On right $S$-Noetherian rings and $S$-Noetherian modules}. Comm. Algebra.  \textbf{46}(2018), 863-869. Zbl 1410.16026. MR3764903. https://doi.org/10.1080/00927872.2017.1332199

	\bibitem{rg80}
	R. Gilmer, W. Heinzer :\textit{The Laskerian Property, Power Series Rings and Noetherian Spectra.}  Proc.  Amer.  Math. Soc. \textbf{79}(1980), 13-16. Zbl 0447.13009. MR0560575. https://doi.org/10.1090/S0002-9939-1980-0560575-6

\bibitem{ki74}
I. Kaplansky :\textit{Commutative rings.} University of Chicago Press p.59 (1974). Zbl 0296.13001. 

	\bibitem{le05} 
	 E. Lasker :\textit{Zur Theorie der Moduln und Ideale}. Math. Ann. \textbf{60}(1905), 20--116. Zbl 0146.25803. http://eudml.org/doc/158174
	 
	
	\bibitem{lw15} 
	J. W.  Lim :\textit{A Note on $S$-Noetherian Domains. Kyungpook Math}.  J.  \textbf{55}(2015), 507-514. Zbl 1329.13006. MR3414628. http://dx.doi.org/10.5666/KMJ.2015.55.3.507
	

		\bibitem{me22}
			E. Massaoud :\textit{$S$-primary ideals of a commutative ring}. Comm. Algebra. \textbf{50}(2022), 988-997. Zbl 1481.13007. https://doi.org/10.1080/00927872.2021.1977939
	
		\bibitem {ne21} 
		E. Noether :\textit{Idealtheorie in Ringbereichen}. Math. Ann. \textbf{83}(1921), 24--66.  Zbl 48.0121.03. MR1511996. https://doi.org/10.1007/BF01464225
		
			\bibitem{ts23}
		T. Singh, A. U. Ansari, S. D. Kumar :\textit{$S$-Noetherian Rings, Modules and their generalizations.} Surv. Math. Appl. \textbf{18}(2023), 163-182. https://www.utgjiu.ro/math/sma/v18/v18.html
		
%

	\bibitem{vd17}
	D. C. Vella :\textit{Primary Decomposition in Boolean Rings}. Pi Mu Epsilon Journal.  \textbf{14}(2018), 581-588.
	https://doi.org/10.48550/arXiv.1707.07783
	
\end{thebibliography}
\end{document}